\documentclass[11pt,reqno]{amsart}
\usepackage[utf8]{inputenc} 
\usepackage[T1]{fontenc}    
\usepackage[english]{babel} 
\usepackage{amsmath}
\usepackage{amsthm}
\usepackage{amsfonts}
\usepackage{amsbsy,amsfonts,amsmath,amssymb,amscd,amsthm}
\usepackage{mathrsfs,float}
\usepackage[mathcal]{euscript}
\usepackage{enumitem}
\usepackage{mathtools}
\usepackage[colorlinks=true,linkcolor=blue,citecolor=green]{hyperref}
\usepackage{tikz}
\usepackage{tikz-cd}

\numberwithin{equation}{section}

\usepackage[a4paper,text={6.1in,8in},centering,top=1.5in]{geometry}
\small\normalsize
\newcommand{\proj}{P}
\newcommand{\dproj}{d_{\mathrm{proj}}}
\newcommand{\RR}{\mathbb R}

\newcommand{\cT}{\mathcal T}
\newcommand{\supp}{\operatorname{supp}}

\newcommand{\Gr}{\operatorname{Gr}}

\newcommand{\one}{\mathbf 1}

\newcommand{\cK}{\mathcal K}

\newcommand{\dd}{\,d}
\newcommand{\GL}{\mathrm{GL}}

\newcommand{\Leb}{\operatorname{Leb}}
\newcommand{\Pcpt}{\mathcal P_{\mathrm{cpt}}}
\newcommand{\Chi}{\mathcal X}

\theoremstyle{plain}
\newtheorem{mainthm}{Theorem}

\newtheorem{maincor}[mainthm]{Corollary}

\newtheorem{theorem}{Theorem}[section]
\newtheorem{proposition}[theorem]{Proposition}
\newtheorem{lemma}[theorem]{Lemma}
\newtheorem{corollary}[theorem]{Corollary}
\newtheorem{claim}{Claim}
\theoremstyle{definition}

\theoremstyle{remark}

\title[Pointwise H\"older continuity of Lyapunov exponents]{Pointwise H\"older continuity of Lyapunov exponents of random matrix products}

\author[Barrientos]{Pablo G. Barrientos}
\address{Instituto de Matem\'atica e Estat\'istica, UFF, Niter\'oi, Brazil}
    \email{pgbarrientos@id.uff.br}
\date{}

\begin{document}
~\vspace{-1cm}
\begin{abstract}
We prove that if the first Lyapunov exponent of a compactly supported random
matrix product is simple, then it is pointwise H\"older continuous for the Wasserstein--Hausdorff distance. By passing to
exterior powers, we obtain pointwise H\"older continuity of the first $k$
exponents whenever they are simple. \end{abstract}
~\vspace{-2cm}
\maketitle

\section{Introduction}\label{sec:intro} \thispagestyle{empty} 
Avila, Eskin and Viana \cite{AEV} proved continuity of all Lyapunov
exponents of compactly supported random matrix products when the laws
converge weakly and their supports converge in the Hausdorff topology. In this paper we prove a pointwise H\"older estimate for the exponents at every law
where they are simple. 

Let $d\ge2$, equip $\RR^d$ with a Euclidean norm, put $G=\GL(d,\RR)$, and denote by $\Pcpt(G)$ the space of compactly supported Borel probability
measures on $G$. For $\nu\in\Pcpt(G)$, define inductively the Lyapunov exponents $\lambda_1(\nu)\ge\cdots\ge\lambda_d(\nu)$ by
\[
 \sum_{i=1}^j\lambda_i(\nu)
 =\lim_{n\to\infty}\frac1n\int_{G^n}
       \log\|\wedge^j(g_n\cdots g_1)\|\,d\nu^{\otimes n},
 \qquad 1\le j\le d.
\]
Equip $G$ with $ \delta(g,h)=\|g-h\|+\|g^{-1}-h^{-1}\|$. 
Let $\delta_W$ be its $1$-Wasserstein distance and $\delta_H$ the associated
Hausdorff distance on compact subsets of $G$. Set
\begin{equation*}\label{eq:deltaT}
 \delta_{\cT}(\nu,\mu)
 =\delta_W(\nu,\mu)+\delta_H(\supp\nu,\supp\mu).
\end{equation*}
This metric induces the topology of \cite{AEV}. 
%Hausdorff convergence of
%supports confines them to a common compact subset of $G$, on which
%Wasserstein convergence is equivalent to weak convergence.

\begin{mainthm}\label{thm:main}
Let $\nu_0\in\Pcpt(G)$ satisfy $\lambda_1(\nu_0)>\lambda_2(\nu_0)$. Then there exist $r_0>0$, $C>0$ and $\theta\in(0,1]$ such that
\[
 |\lambda_1(\nu)-\lambda_1(\nu_0)|
 \le C\delta_{\cT}(\nu,\nu_0)^\theta
 \quad\text{for every }\nu\in\Pcpt(G)
 \text{ with }\delta_{\cT}(\nu,\nu_0)<r_0.
\]
\end{mainthm}

Applying Theorem~\ref{thm:main} to exterior powers gives the following pointwise
H\"older continuity of the first $k$ exponents. 

\begin{maincor}\label{cor:first-k}
Let $\nu_0\in\Pcpt(G)$  satisfy
$\lambda_j(\nu_0)>\lambda_{j+1}(\nu_0)$
for $j=1,\dots,\min\{k,d-1\}$ and some $k\in \{1,\dots,d\}$. Then there exist $r_0>0$, $C>0$ and $\theta\in(0,1]$ such that
\[
 |\lambda_i(\nu)-\lambda_i(\nu_0)|
 \le C\delta_{\cT}(\nu,\nu_0)^\theta \quad \text{for every $\nu\in\Pcpt(G)$ with $\delta_{\cT}(\nu,\nu_0)<r_0$ and  $i=1,\dots, k$.}
\]
In particular, the entire Lyapunov spectrum is pointwise H\"older
continuous at every compactly supported law with simple spectrum.
\end{maincor}

In dimension two, this result was proved by Tall and
Viana~\cite{TV20}. The H\"older estimate compares each perturbation with the fixed law $\nu_0$. A locally uniform H\"older estimate between any two laws in a neighborhood is already known under additional assumptions. Le Page~\cite{LP89} proved this locally uniform 
H\"older regularity of the first Lyapunov exponent assuming proximality and strong irreducibility. Barrientos and Malicet
\cite[Prop.~XIV(3)]{BMbook} improve the result of Le Page by proving locally uniform H\"older continuity under simplicity of the top exponent
in arbitrary dimension for compactly supported almost-irreducible
cocycles, including the case where the transposed cocycle is
almost-irreducible. Recall that the equator of a law is its maximal invariant subspace $E$
with $\lambda_1(\nu|_E)<\lambda_1(\nu)$, where $\nu|_E$ is the law of the
restrictions and $\lambda_1(\nu|_{\{0\}})=-\infty$. The law $\nu$ is \emph{quasi-irreducible}
when $E=\{0\}$. \emph{Almost-irreducibility} means that there are a nonempty compact set
$X\subset\proj(\RR^d)\setminus\proj(E)$ and an open set $O\subset X$ in $\proj(\RR^d)$ such that
$gX\subset O$ for all  $g \in \supp\nu$,   \cite[Def.~1.10.2]{BMbook}. In particular, every quasi-irreducible law is almost-irreducible. Moreover, the classical assumptions of irreducibility and strong irreducibility imply almost-irreducibility.

For laws without compact support, Barrientos, Nisoli and Malicet \cite[Thm,~C]{BNM} prove continuity under quasi-irreducibility and a
logarithmic moment assumption, with respect to an averaged distance, or
equivalently its associated Wasserstein distance. Under simplicity and an
exponential moment assumption, and assuming that the law or its transpose
is quasi-irreducible, Barrientos and Malicet
\cite[Prop.~XII and XIV(2)]{BMbook} obtain locally uniform H\"older
continuity in  Wasserstein distances. In these
statements, the Wasserstein formulation is on the corresponding space of
laws with finite first moment for the specified cost $\delta$.

The starting point here is the projective dynamics on
$X=\proj(\RR^d)\setminus\proj(E)$. Under simplicity of the top exponent,
Barrientos and Malicet \cite[Thm.~C]{BMlc} prove contraction on average
in a weighted projective geometry and a spectral gap on a weighted H\"older space. Their repulsion estimate for the equator supplies a
Lyapunov function. Together with the Poisson equation on the quotient,
this yields global regularizations of a logarithmically singular
coboundary. These functions control both the exponent and stationary mass
in shrinking neighborhoods of the equator. Arbitrary perturbations need
not preserve the equator or an attracting compact subset of $X$.
To treat this difficulty, we use the localized induction of
\cite[Thm.~7.1]{AEV}, after selecting an annulus and conjugating its
outer boundary to a fixed scale.

Section~\ref{sec:fixed-base} contains the preliminary estimates, and
Section~\ref{sec:geometry} proves the upper bound of the H\"older estimate.
Section~\ref{sec:bad-sequence} establishes the lower bound by contradiction.
Section~\ref{sec:conclusion} completes the proof of Theorem~\ref{thm:main} and proves Corollary~\ref{cor:first-k}.

\section{One-step cocycles and preliminary estimates}\label{sec:fixed-base}
\subsection{Realization on a common probability space}\label{subsec:fixed-base-roadmap}
We use the Bernoulli space
\[
 (\Omega,\mathbb P)=([0,1]\times G,\Leb\times\nu_0)^{\mathbb N},
 \qquad A(\omega)=g_0,\quad \omega_0=(u_0,g_0).
\]
For a one-step cocycle $B$ on this space, write $\nu_B=B_*\mathbb P$ and $\lambda_1(B)=\lambda_1(\nu_B)$. In particular, $\nu_A=\nu_0$ and $\lambda_1(A)=\lambda_1(\nu_0)$. For any two such cocycles, use
\[
 d(A,B)=\int\delta(A,B)\,d\mathbb P,\qquad
 \Delta(A,B)=d(A,B)+\delta_H(\supp\nu_A,\supp\nu_B).
\]
The coupling realization lemma \cite[Lem.~7.6.2]{BMbook} gives, for every
$\nu\in\Pcpt(G)$ and $\zeta>0$, a one-step cocycle $B:\Omega \to G$ with
\begin{equation}\label{eq:BM-realization}
 \nu_B=\nu,\qquad
 \Delta(A,B)\le\delta_{\cT}(\nu_0,\nu)+\zeta.
\end{equation}
Conversely, the joint
law of $(A,B)$ is a coupling, so by the definition of Wasserstein distance, 
\begin{equation}\label{eq:law-nonexpansion}
 \delta_{\cT}(\nu_0,\nu_B)\le\Delta(A,B).
\end{equation}
See~\cite[Rem.~7.6.3]{BMbook} for details. It therefore suffices to prove, for this fixed $A$,
\begin{equation}\label{eq:fixed-reference-target}
 |\lambda_1(B)-\lambda_1(A)|\le C\Delta(A,B)^\theta
\end{equation}
for all sufficiently close $B$. The constants will depend only on $\nu_0$.

If $A$ is quasi-irreducible, the conclusion follows from
\cite[Prop.~XII and~XIV(2)]{BMbook}. Henceforth we assume that the equator $E$ is non-trivial. Throughout the proof, the constants $C$ or $C_n$ denote positive constants depending only on the
reference law and on parameters already fixed; their values may increase
between estimates.

\subsection{Products, potentials and Markov operators}
Write $A^n(\omega)=A(\omega_{n-1})\cdots A(\omega_0)$ for $n\ge1$; use the analogous notation for $B^n$.
For every Euclidean vector space $W$, denote its projective space by
$\proj(W)$ and equip it with the angular distance
\[
 \dproj([x],[y]):=\frac{\|x\wedge y\|}{\|x\|\,\|y\|},
 \qquad [x],[y]\in\proj(W).
\]
For $0<\alpha\le1$ and  $f\in C^\alpha(\proj(\RR^d))$ we use the norm
$\|f\|_{\alpha}=\|f\|_\infty+[f]_\alpha$, where
\[
 [f]_\alpha:=\sup_{[x]\ne[y]}
 \frac{|f([x])-f([y])|}{\dproj([x],[y])^\alpha}.
\]
The Markov operator and the projective potential are 
\[
 P_A f([x])=\int f([A(\omega)x])\,\dd\mathbb{P} \quad \text{and} \quad
 \phi_A([x])=\int\log\frac{\|A(\omega)x\|}{\|x\|} \dd\mathbb{P}.
\]
Put
\begin{equation}\label{eq:phi-def}
 \phi_{A,n}([x])
 :=\int\log\frac{\|A^n(\omega)x\|}{\|x\|}\,\dd\mathbb P
 =\sum_{j=0}^{n-1}P_A^j\phi_A([x]).
\end{equation}

After restricting to a sufficiently small $\Delta$-neighborhood of $A$, all
one-step supports are contained in one compact set
$\cK\subset G$.  Hence there is $\Lambda\ge1$ such that
\begin{equation}\label{eq:Lambda}
 \max\{\|g\|,\|g^{-1}\|\}\le\Lambda
 \quad \text{for all $g\in\cK$}.
\end{equation}

\begin{lemma}\label{lem:product-pert}
For every fixed $n\ge1$ there is $C_n>0$ such that
\begin{equation}\label{eq:product-mean}
 \int\|B^n-A^n\|\,\dd\mathbb P\le C_n d(A,B).
\end{equation}
\end{lemma}

\begin{proof}
Put $a_j=A(\omega_{j-1})$ and $b_j=B(\omega_{j-1})$, $1\le j\le n$.
Then $A^n=a_n\cdots a_1$ and $B^n=b_n\cdots b_1$, and
\[
 B^n-A^n
 =\sum_{j=1}^n b_n\cdots b_{j+1}(b_j-a_j)a_{j-1}\cdots a_1.
\]
By \eqref{eq:Lambda},
 $\|B^n-A^n\|\le\Lambda^{n-1}\sum_{j=1}^n\|b_j-a_j\|$. Since all coordinates have the same law, integrating this yields~\eqref{eq:product-mean}. 
\end{proof}

\begin{lemma} \label{lem:fixed-pert}
Fix $n\ge1$ and $0<\alpha\le1$.  There is $C_n$ such that, for all $B$ in the
chosen neighborhood and $\delta=\Delta(A,B)$,
\begin{align}
 \|\phi_{B,n}-\phi_{A,n}\|_\infty
 \le C_n\delta \quad \text{and}  \quad
 \|(P_B^n-P_A^n)f\|_\infty
 \le C_n\delta^\alpha\|f\|_{\alpha}.          \label{eq:fixed-operator}
\end{align}
\end{lemma}

\begin{proof}
Let $\cK_n$ be the compact set of all products of $n$ matrices from $\cK$.
For $g,h\in\cK_n$ and a unit vector $x$, \eqref{eq:Lambda} gives
$\|gx\|,\|hx\|\ge\Lambda^{-n}$.  Since $|\log s-\log t|\le
\Lambda^n|s-t|$ for $s,t\ge\Lambda^{-n}$,
\[
 \left|\log\|gx\|-\log\|hx\|\right|
 \le \Lambda^n\|(g-h)x\|
 \le \Lambda^n\|g-h\|.
\]
After integrating with $g=B^n$, $h=A^n$ and using
Lemma~\ref{lem:product-pert}, we obtain the first inequality in \eqref{eq:fixed-operator} uniformly
in $[x]$. For the projective action, using unit representatives and the elementary
estimate for normalized vectors and angular distance,
\[
\dproj(g[x],h[x])\le \left\|\frac{gx}{\|gx\|}-\frac{hx}{\|hx\|}\right\|
 \le 2\Lambda^n\|(g-h)x\|
 \le 2\Lambda^n\|g-h\|.
\]
Consequently, 
$ |f(B^n[x])-f(A^n[x])|
 \le [f]_{\alpha}\,(2\Lambda^n)^\alpha\|B^n-A^n\|^\alpha$. 
Integrating and using the concavity of $s\mapsto s^\alpha$ for $0<\alpha\le1$ gives
\[
 \|(P_B^n-P_A^n)f\|_\infty \leq [f]_\alpha (2\Lambda^n)^\alpha \int\|B^n-A^n\|^\alpha\dd\mathbb P
 \le C_n \|f\|_\alpha  d(A,B)^\alpha. \qedhere
\]
\end{proof}

We will use the Furstenberg variational formula~\cite[Thm.~5.2.1]{BMbook}
\[
 \lambda_1(B)=\max\left\{\int\phi_B\,dm:P_B^*m=m\right\}.
\] 
A stationary probability attaining the maximum is called \emph{maximizing}.
For such a probability, stationarity and \eqref{eq:phi-def} give
$\int\phi_{B,n}\,dm=n\lambda_1(B)$ for every $n\ge1$.

\subsection{The equator, the quotient and the Poisson equation}\label{sec:reference-geometry}

Let $E\subset\RR^d$ be the equator of~$A$. By assumption $E\ne\{0\}$, and $\lambda_1(A|_E)<\lambda_1(A)$ by definition.
Set
\[
 X:=\proj(\RR^d)\setminus\proj(E) \qquad \text{and} \qquad 
  \lambda:=\lambda_1(A).
\]
Fix the orthogonal decomposition $\RR^d=E\oplus E^\perp$ and set $V:=\RR^d/E$. Let $Q:\RR^d\to V$ be the quotient map, and identify $V$ isometrically
with $E^\perp$. Set
\begin{equation}\label{eq:r}
 r([x]):=\dproj([x],\proj(E)),\qquad [x]\in\proj(\RR^d).
\end{equation}
The function $r$ is $1$-Lipschitz because it is the distance to the closed set
$\proj(E)$.  For $x\ne0$, orthogonal projection gives the exact formula $r([x])=\|Qx\| / \|x\|$. 
Put
\[
 \ell:=\log r\quad\hbox{on }X.
\]

Let $\bar A$ be the quotient cocycle on $V$.  According to~\cite[Claim 5.5.2]{BMbook},  $\bar A$ is~quasi-irreducible~and 
\begin{equation}\label{eq:quotient-top}
 \lambda_1(\bar A)=\lambda_1(A)=\lambda.
\end{equation}
For $x\notin E$, the identity
$\frac{r([x])}{r([A^nx])}
 =\frac{\|A^nx\|/\|x\|}{\|\bar A^nQx\|/\|Qx\|}$
gives, after taking logarithms and expectations,
\begin{equation}\label{eq:radial-identity}
 \phi_{A,n}([x])
 =\bar\phi_{A,n}([Qx])+\ell([x])-P_A^n\ell([x]),
 \qquad [x]\in X.
\end{equation}
Here $\bar \phi_{A,n}$ denotes~\eqref{eq:phi-def} for the cocycle $\bar A$.  
This is the identity used in \cite[Prop.~7.1]{BMlc}.

\begin{lemma}\label{lem:quotient-poisson}
There exist $0<\alpha\le1$ and
$\bar U\in C^\alpha(\proj(V))$ such that
\begin{equation}\label{eq:quotient-poisson}
 \bar\phi_A-\lambda=\bar U-P_{\bar A}\bar U.
\end{equation}
Consequently, with $u([x]):=\bar U([Qx])$ and $H([x]):=u([x])+\ell([x])$  for $[x] \in X$,  one has
\begin{equation}\label{eq:exact-coboundary}
 \phi_{A,n}-n\lambda=H-P_A^nH\quad\text{on $X$ for every $n\ge1$.}
\end{equation}

\end{lemma}

\begin{proof}
If $\dim V=1$, then $\proj(V)$ is a point,
$\bar\phi_A\equiv\lambda$, and we take $\bar U=0$; then
\eqref{eq:radial-identity} immediately gives \eqref{eq:exact-coboundary}.
Assume $\dim V\ge2$. Since $Q$ is the orthogonal quotient,
$\|\wedge^2\bar A^n\|\le\|\wedge^2 A^n\|$.
Taking logarithms, expectations and limits gives
$\lambda_1(\bar A)+\lambda_2(\bar A)
 \le\lambda_1(A)+\lambda_2(A)$. 
Together with \eqref{eq:quotient-top}, this proves
$\lambda_2(\bar A)\le\lambda_2(A)<\lambda_1(\bar A)$.

Applying  \cite[Prop.~VI]{BMbook},  
$\bar A$ has spectral gap on $C^\alpha(\proj(V))$ for any $\alpha>0$ small enough. That is, there are $C_{\rm sg}>0$,
$\rho_{\rm sg}\in(0,1)$, and a unique stationary probability $\bar m$ such
that
\begin{equation}\label{eq:quotient-spectral-gap}
 \left\|P_{\bar A}^n f-\int f\,d\bar m\cdot\mathbf 1\right\|_{\alpha}
 \le C_{\rm sg}\rho_{\rm sg}^n\|f\|_{\alpha}
 \qquad \text{for every $n\ge0$}.
\end{equation}
Compact support makes $\bar\phi_A$ Lipschitz, by the same elementary estimate
used in Lemma~\ref{lem:fixed-pert}. In particular,
$\bar\phi_A\in C^\alpha(\proj(V))$.

To identify the mean, use \cite[Thm.~5.2.1]{BMbook}, which
states the Furstenberg variational formula
$\lambda_1(\bar A)
 =\max\left\{\int\bar\phi_A\,d\mu:
        \mu\text{ is }P_{\bar A}\text{-stationary}\right\}$. Since $\bar m$ is the unique stationary probability and
$\lambda_1(\bar A)=\lambda$ by \eqref{eq:quotient-top}, we get  $\int\bar\phi_A\,d\bar m=\lambda$. 
The spectral gap therefore makes the series
\[
 \bar U:=\sum_{j=0}^{\infty}P_{\bar A}^j(\bar\phi_A-\lambda)
\]
converge in $C^\alpha(\proj(V))$ for some $\alpha>0$.  Telescoping the
series gives~\eqref{eq:quotient-poisson}. Moreover, telescoping on $\proj(V)$ gives
\[
 \bar\phi_{A,n}-n\lambda
 =\sum_{j=0}^{n-1}P_{\bar A}^j(\bar\phi_A-\lambda)
 =\bar U-P_{\bar A}^n\bar U.
\]
Since $QA^n=\bar A^nQ$, composition with $[x]\mapsto[Qx]$ yields
\[
 \bar\phi_{A,n}([Qx])-n\lambda=u([x])-P_A^nu([x]),\qquad [x]\in X.
\]
Combining this identity with the radial identity \eqref{eq:radial-identity} proves
\eqref{eq:exact-coboundary}.
\end{proof}

\begin{lemma}\label{lem:drift}
There exist $N_0\ge1$, $p\in(0,1]$, $q_0\in(0,1)$, and $C_0>0$ such that
\begin{equation}\label{eq:drift}
 P_A^{N_0}R\le q_0R+C_0 \quad \text{with} \quad R:=r^{-p}\quad\text{on }X.
\end{equation}
Moreover, for every $q>0$ there are a multiple $N$ of $N_0$ and a
constant $C_q$ such that
\begin{equation}\label{eq:drift-small}
 P_A^NR\le qR+C_q \quad\text{on }X.
\end{equation}
\end{lemma}

\begin{proof}
By \cite[Prop.~7.1]{BMlc}, $\proj(E)$ is repelling. Compact support
ensures the moment hypothesis of \cite[Prop.~5.3]{BMlc}, so there are 
numbers $N_0\ge1$, $\gamma>0$, $\bar q\in(0,1)$ and $\bar C\ge0$ such that $P_A^{N_0}W\le \bar q W+\bar C$ where $W:=r^{-\gamma}$. Choose $0<p\le\min\{\gamma,1\}$ and put $\vartheta=p/\gamma\in(0,1]$.
Then $R=W^\vartheta$.  Since $s\mapsto s^\vartheta$ is concave, Jensen's
inequality and $(a+b)^\vartheta\le a^\vartheta+b^\vartheta$ give
$P_A^{N_0}R
 =P_A^{N_0}(W^\vartheta)
 \le (P_A^{N_0}W)^\vartheta 
 \le(\bar qW+\bar C)^\vartheta
 \le\bar q^\vartheta W^\vartheta+\bar C^\vartheta$.
Thus \eqref{eq:drift} holds with
$q_0=\bar q^\vartheta\in(0,1)$ and $C_0=\bar C^\vartheta$. Finally, iteration gives
\[
 P_A^{kN_0}R
 \le q_0^kR+C_0\sum_{j=0}^{k-1}q_0^j
 \le q_0^kR+\frac{C_0}{1-q_0}.
\]
Given $q>0$, choose $k$ so that $q_0^k\le q$.  Taking $N=kN_0$ and
$C_q=\frac{C_0}{1-q_0}$ proves~\eqref{eq:drift-small}.
\end{proof}
\section{Global regularizations and the upper bound}\label{sec:geometry}

The quotient spectral gap holds for all sufficiently small H\"older exponents.
Choose its exponent so that
$0<\alpha\le\min\{p,1/2\}$, where $0<p\le1$ was given in Lemma~\ref{lem:drift}. 

Lemma~\ref{lem:quotient-poisson} gives the coboundary $H=u+\log r$ on $X$.
Although $u$ is bounded, $H$ diverges at the equator. We construct global
H\"older regularizations of $H$ and use them to prove the upper estimate
$\lambda_1(B)-\lambda\le C\Delta(A,B)^{\theta_+}$.

\subsection{Regular extensions}  Since $\bar U\in C^\alpha(\proj(V))$ and $\proj(V)$ is compact, 
$u([x])=\bar U([Qx])$, for $[x]\in X$,
is bounded on $X$. In fact,
$\|u\|_\infty:= \sup_X |u| =\sup_{\proj(V)}|\bar U| <\infty$.
Since the reference cocycle $A$ is fixed throughout the proof, so are
$\bar U$ and $u$. Thus, the constants below may depend on
$[\bar U]_{\alpha}$ and $\|u\|_{\infty}$, but
they are independent of the perturbation $B$ and of the scale parameters.

\begin{lemma}\label{lem:quotient-holder}
There is $C>0$ such that for every $0<\varepsilon<1$
\begin{equation}\label{eq:u-holder-away}
 |u([x])-u([y])|
 \le C\varepsilon^{-2\alpha}\dproj([x],[y])^\alpha \quad \text{for any $[x],[y]\in X$ with $r([x]),r([y])\ge\varepsilon$}.
\end{equation}
\end{lemma}

\begin{proof}
Choose unit representatives $\|x\|=\|y\|=1$.  Then
$\|Qx\|=r([x]), \|Qy\|=r([y])\ge\varepsilon$.  
For the angular metric, $\dproj([Qx],[Qy]) = \|Qx\wedge Qy\| / (\|Qx\|\,\|Qy\|)$, cf.~\cite[Lem.~5.6.1]{BMbook}.  
The linear map $\wedge^2Q$ is bounded, so
$\|Qx\wedge Qy\|\le C\|x\wedge y\| \leq C \dproj([x],[y])$ for normalized representatives, cf.~\cite[Sec.~5.6]{BMbook}.   
Division by the two
lower bounds $\varepsilon$ proves %\begin{equation*}%\label{eq:quotient-lip}
$\dproj([Qx],[Qy])
 \le C\varepsilon^{-2}\dproj([x],[y])$.
Applying this with the
$\alpha$-H\"older continuity of $\bar U$, we obtain~\eqref{eq:u-holder-away}.
\end{proof}
For $s>0$ define on $\proj(\RR^d)$
\begin{equation}\label{eq:ells}
 \ell_s([x]):=\frac1p\log\bigl(r([x])^p+s^p\bigr).
\end{equation}
Observe that $\ell_s=\log s$ on $\proj(E)$. 
\begin{lemma}\label{lem:reg-log}
There is $C>0$ such that  for any $0<s<1$,  $\ell_s\in C^\alpha(\proj(\RR^d))$  and $\|\ell_s\|_{\alpha}\le Cs^{-1}$.
\end{lemma}

\begin{proof}
Continuity at $\proj(E)$ follows immediately from \eqref{eq:ells}.  Since
$0<p\le1$, we have $|a^p-b^p|\le|a-b|^p$ for any $a,b\ge0$. Also, on $[s^p,\infty)$ the derivative of $\log$ is at most $s^{-p}$.
Therefore, using the Lipschitz continuity of $r$,
\[
 |\ell_s([x])-\ell_s([y])|
 \le \frac1p s^{-p}|r([x])^p-r([y])^p|
 \le Cs^{-p}\dproj([x],[y])^p.
\]
Because $\alpha\le p$ and $\proj(\RR^d)$ has bounded diameter, this implies an
$\alpha$-H\"older seminorm bounded by $Cs^{-p}$.
Since $0\le r\le1$,
\[
 \log s\le\ell_s\le\frac1p\log(1+s^p)\le\frac{\log2}{p}.
\]
Thus $\|\ell_s\|_\infty\le |\log s|+p^{-1}\log2\le Cs^{-p}$,
where we used $-\log s\le p^{-1}s^{-p}$ for $0<s<1$.
Combining both bounds gives
$\|\ell_s\|_\alpha\le Cs^{-p}\le Cs^{-1}$, since $0<p\le1$.
\end{proof}

Fix once and for all a smooth $\chi:[0,\infty)\to[0,1]$ with
\[
 \chi=0\ \text{on }[0,2],\qquad \chi=1\ \text{on }[3,\infty).
\]
For $\varepsilon>0$ set $\chi_\varepsilon([x])=\chi(r([x])/\varepsilon)$ for $[x]\in \proj(\RR^d)$. Define 
$$
   u_\varepsilon = \chi_\varepsilon u \ \ \text{on $X=P(\mathbb R^d)\setminus \proj(E)$ \quad and
   }\quad u_\varepsilon =0  \ \ \text{on $\proj(E)$}.
$$
\begin{lemma} \label{lem:cutoff-u} There is $C>0$ such that for any $\varepsilon >0$, $u_\varepsilon \in C^\alpha(\proj(\RR^d))$ and $
 \|u_\varepsilon\|_{\alpha}
 \le C\varepsilon^{-1}$. 
\end{lemma}

\begin{proof}
For $\varepsilon\ge1/2$, $u_\varepsilon=0$ because $r\le1$. Assume $0<\varepsilon<1/2$. The extension $u_\varepsilon$ is continuous because $\chi_\varepsilon$ vanishes on the 
neighborhood $\{r\le2\varepsilon\}$.  
Since the Lipschitz norm of $\chi_\varepsilon$ is at most
$C\varepsilon^{-1}$, we have that  if $r([x]),r([y])>2\varepsilon$, then Lemma~\ref{lem:quotient-holder} gives
\[
\begin{aligned}
 |u_\varepsilon([x])-u_\varepsilon([y])|
 &\le \|\chi_\varepsilon\|_\infty |u([x])-u([y])|
      +\|u\|_\infty
       |\chi_\varepsilon([x])-\chi_\varepsilon([y])|\\
 &\le C\bigl(\varepsilon^{-2\alpha}+\varepsilon^{-1}\bigr)
       \dproj([x],[y])^\alpha
\end{aligned}
\]
where in the second term we used boundedness of $\proj(\RR^d)$ to replace
$\dproj$ by a fixed constant times $\dproj^\alpha$.
Suppose next that $r([x])\le2\varepsilon<r([y])$.  Then
$u_\varepsilon([x])=\chi_\varepsilon([x])=0$ and thus 
\[
 |\chi_\varepsilon([y])|=|\chi_\varepsilon([y])-\chi_\varepsilon([x])|
  \le C\varepsilon^{-1}\dproj([x],[y]).
\]
Thus $|u_\varepsilon([y])-u_\varepsilon([x])|\leq \|u\|_\infty |\chi_\varepsilon([y])|  
 \le C\varepsilon^{-1}\dproj([x],[y])^\alpha$.
The case with $[x],[y]$ interchanged is identical, and if both points
lie in $\{r\le2\varepsilon\}$ the difference is zero.  Hence
$[u_\varepsilon]_\alpha\le
C(\varepsilon^{-2\alpha}+\varepsilon^{-1})$, while
$\|u_\varepsilon\|_\infty\le\|u\|_\infty$.  Since $2\alpha\le1$ and $0<\varepsilon<1$, both terms are bounded
by $C\varepsilon^{-1}$ completing the proof.
\end{proof}

\subsection{A regularized coboundary inequality}
\label{sec:upper-half}

Choose a number $a>3$, to be fixed quantitatively in the proof below, and put
\begin{equation}\label{eq:h-eps-new}
 h_\varepsilon:= u_\varepsilon+\ell_{a\varepsilon}
 \qquad \text{for $0<\varepsilon<1/a$}.
\end{equation}
By Lemmas~\ref{lem:reg-log} and \ref{lem:cutoff-u}, $h_\varepsilon\in C^\alpha(\proj(\RR^d))$ and
\begin{equation}\label{eq:h-eps-kappa}
 \|h_\varepsilon\|_{\alpha}\le C\varepsilon^{-1}.
\end{equation}

\begin{proposition} \label{prop:barrier-reference-new}
There are  $N\ge1$, $C>0$, 
$a>3$ and $\varepsilon_0\in(0,1/a)$ such that 
\begin{equation}\label{eq:barrier-reference-new}
 \frac1N\phi_{A,N}-\lambda
 \le \frac1N(h_\varepsilon-P_A^Nh_\varepsilon)
     -\frac1N\one_{\{r\le3\varepsilon\}}+C\varepsilon^p
 \quad\text{on }\proj(\RR^d) \ \ \text{for every
$\varepsilon\in (0,\varepsilon_0)$}.
\end{equation}
\end{proposition}

\begin{proof}
On $X$ put
\[
 G_\varepsilon:=(1-\chi_\varepsilon)u,
 \qquad
 L_{a\varepsilon}:=\ell_{a\varepsilon}-\ell
 =\frac1p\log\bigl(1+(a\varepsilon)^pR\bigr),
\]
where $R=r^{-p}$ is the Lyapunov function in Lemma~\ref{lem:drift}. Since $H-h_\varepsilon=G_\varepsilon-L_{a\varepsilon}$,
the exact coboundary identity gives
\begin{equation}\label{eq:barrier-exact-new}
 \phi_{A,N}-N\lambda
 =h_\varepsilon-P_A^Nh_\varepsilon
 +(G_\varepsilon-P_A^NG_\varepsilon)
 -(L_{a\varepsilon}-P_A^NL_{a\varepsilon}).
\end{equation}

Set $M_u:=\max\{1,\|u\|_\infty\}$ and, for $0<q<1$, define
\[
 J_q(z):=\frac1p\log\frac{1+z}{1+qz},\qquad z\ge0.
\]
\begin{claim} \label{eq:J-inner-new} There are  $a>3$ and $q \in (0,1)$ such that 
\begin{align*}
 M_u\min\{1, (3/a)^p qz\} \le J_q(z)\quad \text{for $z\ge0$} \quad  \text{and} \quad 
 J_q(z)\ge2M_u+1\quad \text{for $z\ge (a/3)^p $}.
\end{align*}
\end{claim}
\begin{proof} Direct differentiation gives 
\[
 J_q'(z)=\frac{1-q}{p(1+z)(1+qz)}>0,\qquad
 J_q''(z)=-\frac{(1-q)(1+q+2qz)}{p(1+z)^2(1+qz)^2}<0.
\]
Thus $J_q$ is increasing and concave with $J_q(0)=0$.  Hence, for $z\geq (a/3)^p$,  
\[
 J_q(z)\geq J_q((a/3)^p)\longrightarrow\frac1p\log(1+(a/3)^p)
 \qquad \text{as $q\downarrow0$}.
\]
Choosing $a>3$ such that $\frac1p \log(1+(a/3)^p)>2M_u+2$ we get the second inequality for small $q>0$.
To show the first inequality, put $Z_q=(c_aq)^{-1}$ with $c_a=(3/a)^p$ and note that
\[
 J_q(Z_q)=\frac1p\log\frac{1+(c_aq)^{-1}}{1+c_a^{-1}}
 \longrightarrow\infty \qquad \text{as $q\downarrow0$}.
\]
Decrease $q>0$ so that $J_q(Z_q)\ge M_u$. By concavity on $[0,Z_q]$
and monotonicity on $[Z_q,\infty)$, 
\[
 J_q(z)\ge\frac z{Z_q}J_q(Z_q)\ge M_uc_aqz,
 \qquad
 J_q(z)\ge J_q(Z_q)\ge M_u,
\]
which proves the first inequality and concludes the proof of the claim.
\end{proof}
By Lemma~\ref{lem:drift}, we can choose a multiple $N$ of its block length
such that $P_A^NR\le qR+C_q$. For $z=(a\varepsilon)^pR$, Jensen's inequality yields
\[
 P_A^NL_{a\varepsilon}
 \le\frac1p\log\bigl(1+(a\varepsilon)^pP_A^NR\bigr) \le\frac1p\log\bigl(1+qz+C_q(a\varepsilon)^p\bigr)
 \le\frac1p\log(1+qz)+C\varepsilon^p.
\]
Consequently,
$P_A^NL_{a\varepsilon}-L_{a\varepsilon}\le -J_q(z)+C\varepsilon^p$.

On the other hand,
$|G_\varepsilon|\le M_u\one_{\{r\le3\varepsilon\}}$ and
$\one_{\{r\le3\varepsilon\}}\le(3\varepsilon)^pR$, so
\[
 P_A^N\one_{\{r\le3\varepsilon\}}
 \le\min\{1,(3\varepsilon)^pP_A^NR\}\le\min\{1,(3/a)^p qz+C\varepsilon^p\}
 \le\min\{1,(3/a)^p qz\}+C\varepsilon^p.
\]
It follows that $G_\varepsilon-P_A^NG_\varepsilon
 \le M_u\one_{\{r\le3\varepsilon\}}
     +M_u\min\{1,(3/a)^p qz\}+C\varepsilon^p$. Hence
\[
(G_\varepsilon-P_A^NG_\varepsilon)
 -(L_{a\varepsilon}-P_A^NL_{a\varepsilon})
\le M_u\one_{\{r\le3\varepsilon\}}
 +M_u\min\{1,(3/a)^p qz\}-J_q(z)+C\varepsilon^p.
\]
On $\{r>3\varepsilon\}$, the first inequality in Claim~\ref{eq:J-inner-new} bounds the right-hand side by $C\varepsilon^p$. On $\{r\le3\varepsilon\}$, we have
$z=(a\varepsilon/r)^p\ge(a/3)^p$, and now the second inequality in Claim~\ref{eq:J-inner-new} bounds it by
$2M_u-(2M_u+1)+C\varepsilon^p=-1+C\varepsilon^p$. Thus
\begin{equation*}%\label{eq:barrier-sign-new}
 (G_\varepsilon-P_A^NG_\varepsilon)
 -(L_{a\varepsilon}-P_A^NL_{a\varepsilon})
 \le-\one_{\{r\le3\varepsilon\}}+C\varepsilon^p
 \qquad\text{on }X.
\end{equation*}
Substituting into \eqref{eq:barrier-exact-new} and dividing by $N$ gives~\eqref{eq:barrier-reference-new} on $X$. Since the function
$\phi_{A,N}-N\lambda-(h_\varepsilon-P_A^Nh_\varepsilon)$
is continuous on $\proj(\RR^d)$, approximation from
$X\cap\{r<3\varepsilon\}$ extends the inequality to $\proj(E)$. 
This proves \eqref{eq:barrier-reference-new} on $\proj(\RR^d)$.
\end{proof}

\begin{corollary}\label{cor:master-upper-new}
Let $B$ be sufficiently close to $A$, put
$\delta:=\Delta(A,B)$, and let $m$ be any $B$-stationary maximizing
projective probability.  Then, for every
$0<\varepsilon<\varepsilon_0$,
\begin{equation}\label{eq:master-upper-new}
 \lambda_1(B)-\lambda
 \le -\frac1N m(\{r\le3\varepsilon\})
      +C\varepsilon^p+C\delta^\alpha\varepsilon^{-1}.
\end{equation}
Consequently,
\begin{equation}\label{eq:upper-holder-new}
 \lambda_1(B)-\lambda\le C\delta^{\theta_+},
 \qquad
 \theta_+:=\frac{\alpha p}{p+1},
\end{equation}
and
\begin{equation}\label{eq:tube-upper-new}
 m(\{r\le3\varepsilon\})
 \le N\bigl(\lambda-\lambda_1(B)\bigr)
      +C\bigl(\varepsilon^p+\delta^\alpha\varepsilon^{-1}\bigr).
\end{equation}
\end{corollary}

\begin{proof}
Integrate
\eqref{eq:barrier-reference-new}.  Add and subtract
$\phi_{B,N}$ and $P_B^Nh_\varepsilon$.  Since $m$ is $P_B^N$-stationary,
\[
 \int(h_\varepsilon-P_B^Nh_\varepsilon)\,dm=0.
\]
Therefore
\[
\begin{aligned}
 \lambda_1(B)-\lambda
 &\le -\frac1N m(\{r\le3\varepsilon\})+C\varepsilon^p  +\frac1N\|\phi_{B,N}-\phi_{A,N}\|_\infty
 +\frac1N\|(P_B^N-P_A^N)h_\varepsilon\|_\infty.
\end{aligned}
\]
Lemma~\ref{lem:fixed-pert} and \eqref{eq:h-eps-kappa} give
\eqref{eq:master-upper-new}.
Dropping the negative term and choosing
$\varepsilon=\delta^{\alpha/(p+1)}$ makes the two positive errors have
the same order $\delta^{\alpha p/(p+1)}$.  This proves
\eqref{eq:upper-holder-new}.  On the other hand, moving the negative mass
term in \eqref{eq:master-upper-new} to the left gives
\eqref{eq:tube-upper-new}.
\end{proof}

\section{The contradiction sequence and stationary mass}
\label{sec:bad-sequence}\label{sec:defect-concentration}

Corollary~\ref{cor:master-upper-new} gives
$\lambda_1(B)\le\lambda+C\delta^{\theta_+}$, where $\delta=\Delta(A,B)$.
If $D:=\lambda-\lambda_1(B)\le0$, then
$0\le\lambda_1(B)-\lambda\le C\delta^{\theta_+}$, so the required estimate
already holds. It remains to prove a lower bound when $D>0$.
Suppose, for a contradiction, that no constants $C,\theta,r_0>0$ satisfy
\[
 \lambda_1(B)\ge\lambda-C\Delta(A,B)^\theta
 \qquad\text{whenever }\Delta(A,B)<r_0.
\]

\subsection{The contradiction sequence}

\begin{lemma}\label{lem:superpoly-new}
If the lower H\"older estimate fails at the fixed cocycle $A$, there are
cocycles $B_k$ on the fixed Bernoulli base such that
\begin{equation}\label{eq:superpoly-new}
 \delta_k:=\Delta(A,B_k)\longrightarrow0,
 \qquad
 D_k:=\lambda-\lambda_1(B_k)>0,
 \qquad
 \frac{D_k}{\delta_k^q}\longrightarrow\infty
 \quad\text{for every }q>0.
\end{equation}
\end{lemma}

\begin{proof}
For each integer $k\ge1$, use the failure of the lower estimate with exponent
$1/k$, constant $k$, and radius $1/k$.  Choose $B_k$ so that
\[
 0<\delta_k<1/k,
 \qquad
 D_k>k\delta_k^{1/k}.
\]
Fix $q>0$.  For all large $k$, $1/k<q$ and $0<\delta_k<1$, so
$\delta_k^{1/k}\ge\delta_k^q$.  Therefore
\[
 \frac{D_k}{\delta_k^q}
 >k\delta_k^{1/k-q}\ge k\longrightarrow\infty. \qedhere
\]
\end{proof}

\subsection{Generic approximation}

Following \cite[Sec.~5.1]{AEV}, we call a
probability measure on $G$ \emph{generic} if it gives zero mass to every
proper algebraic subvariety of $G$.  Such a measure automatically has
Zariski-dense support: otherwise its support would be contained in one proper
algebraic subvariety, which would have measure one.

By \cite[Prop.~5.3]{AEV}, every compactly supported law $\nu$
can be approximated by compactly supported generic laws with Zariski-dense
support. Since $\delta_{\mathcal T}$ induces the topology used there,
for every $\varepsilon>0$ such an approximant $\nu'$ can be chosen with
$\delta_{\mathcal T}(\nu',\nu)<\varepsilon$.

\begin{lemma}\label{lem:usc-new}
The map $\nu\mapsto\lambda_1(\nu)$ is upper semicontinuous for
$\delta_{\mathcal T}$ on $\Pcpt(G)$.
\end{lemma}

\begin{proof}
This is a folkloric result. Indeed,  let $F_n(\nu) = \int_{G^n} \log \|g_n \cdots g_1\| \, d\nu(g_1) \cdots d\nu(g_n)$. Since $F_{n+m}(\nu) \le F_n(\nu) + F_m(\nu)$, Fekete's lemma implies $\lambda_1(\nu) = \inf_{n \ge 1} \frac{1}{n} F_n(\nu)$. As the infimum of continuous functions $\nu \mapsto F_n(\nu)$, $\lambda_1$ is upper semicontinuous.
\end{proof}

\begin{lemma}\label{lem:generic-bad-new}
After replacing the sequence in Lemma~\ref{lem:superpoly-new} and relabeling,
we may assume that every law $\nu_k:=(B_k)_*\mathbb P$ is generic and
that
\begin{equation}\label{eq:superpoly-generic-new}
 \frac{D_k}{\delta_k^q}\longrightarrow\infty
 \qquad\text{for every }q>0.
\end{equation}
Moreover,
\begin{equation}\label{eq:Dk-to-zero-new}
 D_k\longrightarrow0.
\end{equation}
\end{lemma}

\begin{proof}
Start from Lemma~\ref{lem:superpoly-new} and write $\nu_k=(B_k)_*\mathbb P$.
By \eqref{eq:law-nonexpansion},
$\delta_{\mathcal T}(\nu_k,\nu_0)\le\delta_k$.  Upper semicontinuity from
Lemma~\ref{lem:usc-new} gives, for each $k$, a radius $\eta_k>0$ such that
\[
 \delta_{\mathcal T}(\mu,\nu_k)<\eta_k
 \quad\Longrightarrow\quad
 \lambda_1(\mu)\le\lambda_1(\nu_k)+\frac14D_k.
\]
By \cite[Prop.~5.3]{AEV}, choose a generic law $\nu_k'$ with
Zariski-dense support such that
\[
 \delta_{\mathcal T}(\nu_k',\nu_k)<\min\{\eta_k,\delta_k/4\}.
\]
Then
\[
 \lambda-\lambda_1(\nu_k')\ge\frac34D_k,
 \qquad
 \delta_{\mathcal T}(\nu_k',\nu_0)\le\frac54\delta_k.
\]

We now use the common-base realization~\cite[Lem.~7.6.2]{BMbook} applied to
$\nu_k'$ and the fixed realization $A$ to obtain a realization $B_k'$ of
$\nu_k'$ on the original Bernoulli base with
\[
 \Delta(A,B_k')
 \le \delta_{\mathcal T}(\nu_0,\nu_k')+\frac14\delta_k.
\]
Consequently, $\Delta(A,B_k')\le\frac32\delta_k$ and  $\lambda-\lambda_1(B_k')\ge\frac34D_k$.
Thus, for every fixed $q>0$,
\[
 \frac{\lambda-\lambda_1(B_k')}{\Delta(A,B_k')^q}
 \ge\frac{3}{4(3/2)^q}\frac{D_k}{\delta_k^q}
 \longrightarrow\infty.
\]
Relabeling $B_k'$ as $B_k$ gives~\eqref{eq:superpoly-generic-new}. 
Finally, \cite[Thm.~B]{AEV} gives qualitative continuity
of the Lyapunov exponents in the weak--Hausdorff compact-support topology.
Since $\delta_{\mathcal T}(\nu_k,\nu_0)\le\delta_k\to0$,
$\lambda_1(B_k)=\lambda_1(\nu_k)\to\lambda$.  Hence $D_k\to0$.  
\end{proof}

By \cite[Prop.~5.2]{AEV}, $\nu_k$ has a unique stationary
probability $m_k$ on $\proj(\RR^d)$. By the variational formula
\cite[Thm.~5.2.1]{BMbook}, $m_k$ is therefore maximizing.
We use these probabilities throughout the rest of the proof. We shall repeatedly use one elementary consequence of
\eqref{eq:superpoly-generic-new}. Namely,  
\begin{equation}\label{eq:superpoly-calculus-new}
 \delta_k^cD_k^{-M}=o(D_k) \quad \text{for every $c>0$ and every $M\ge0$.}
\end{equation}
Indeed, choose $q\in(0,\frac{c}{M+1})$.  Then
$\delta_k^q/D_k\to0$, and
$ \delta_k^c/D_k^{M+1}
 =\delta_k^{c-q(M+1)}
  \left(\frac{\delta_k^q}{D_k}\right)^{M+1}
 \to 0$.

\subsection{Stationary mass near the equator}

Choose once and for all $0<a<\alpha/4$ where $\alpha$ is the exponent fixed in Section~\ref{sec:geometry}. The goal of this subsection is to prove the following result: 

\begin{proposition}\label{prop:defect-mass-new}
There are constants $C>0$ and $c>0$ such that for all sufficiently large $k$,
\begin{equation}\label{eq:radial-comparable-new}
 cD_k\le m_k(\{r\le t\})\le CD_k
 \qquad
 \text{for every $t\in [\delta_k^{2a}, \delta_k^a]$}.
\end{equation}
\end{proposition}

We first estimate the error of a regularization of the coboundary. For $D>0$ small enough and $0<t<1$ define the global H\"older function
\begin{equation}\label{eq:gDt}
 h_{D,t}:= u^{}_{tD^{2/p}}+\ell^{}_{tD^{2/p}},
\end{equation}
where $u_s$ and $\ell_s$ were introduced in Section~\ref{sec:geometry}.  Assuming
$0<tD^{2/p}<1$, Lemmas~\ref{lem:reg-log} and \ref{lem:cutoff-u} give
\begin{equation}\label{eq:gDt-norm}
 \|h_{D,t}\|_{\alpha}
 \le C(tD^{2/p})^{-1}
 =Ct^{-1}D^{-2/p}.
\end{equation}
Now, for $N>0$ as in Proposition~\ref{prop:barrier-reference-new}, we define the regularized coboundary error as
\begin{equation}\label{eq:error-Dt}
 e_{D,t}:=\phi_{A,N}-N\lambda-
          (h_{D,t}-P_A^Nh_{D,t}).
\end{equation}

\begin{lemma}\label{lem:regularized-error-new}
There are constants $C>0$ and $D_0>0$ such that for
$0<D<D_0$ and $0<t<1$,
\begin{align}
 \|e_{D,t}\|_\infty\le C                         \quad \text{and} \quad
 |e_{D,t}([x])|\le CD^2
 \quad\text{whenever }r([x])\ge t.              \label{eq:error-outside-new2}
\end{align}
\end{lemma}

\begin{proof}
Compact support gives $K\ge1$ such that
$\|A^N(\omega)\|,\|A^N(\omega)^{-1}\|\le K$ for $\mathbb P$-a.e.~$\omega$.
Since $E$ is invariant, $QA^N=\bar A^NQ$, and the orthogonal quotient
satisfies $\|\bar A^N\|,\|(\bar A^N)^{-1}\|\le K$. Hence
\[
 K^{-1}\|Qx\|\le\|\bar A^NQx\|\le K\|Qx\|,\qquad
 K^{-1}\|x\|\le\|A^Nx\|\le K\|x\|.
\]
Dividing and using $r([A^Nx])=\|\bar A^NQx\|/\|A^Nx\|$ gives
\begin{equation}\label{eq:radial-distortion-new}
 C_N^{-1}r([x])\le r([A^N(\omega)x])\le C_Nr([x]) \quad 
  \text{for $[x]\in X$ and $\mathbb P$-a.e.~$\omega$,}
\end{equation}
where $C_N=K^2$.
As in the proof of Proposition~\ref{prop:barrier-reference-new}, define on $X$ the functions
\[
 G_\rho:=(1-\chi_\rho)u,
 \qquad
 L_\rho:=\ell_{\rho}-\ell
 =\frac1p\log(1+\rho^pR) \quad \text{where \ \ $\rho:=tD^{2/p}$.}
\]
Since $H-h_{D,t}=G_\rho-L_\rho$, the exact coboundary identity
\eqref{eq:exact-coboundary} gives
\begin{equation}\label{eq:error-LG-new}
 e_{D,t}
 =(G_\rho-P_A^NG_\rho)-(L_\rho-P_A^NL_\rho)
 \qquad\text{on }X.
\end{equation}
For the global bound, $|G_\rho-P^N_AG_\rho|\le 2\|u\|_\infty$, while
\eqref{eq:radial-distortion-new} and $R=r^{-p}$ give
\[
 C_N^{-p}\le\frac{R([A^N(\omega)x])}{R([x])}\le C_N^p,
 \qquad
 C_N^{-p}\le\frac{1+\rho^pR([x])}{1+\rho^pR([A^N(\omega)x])}\le C_N^p.
\]
Taking logarithms yields
$|L_\rho([x])-L_\rho([A^N(\omega)x])|\le\log C_N$. Averaging gives 
$$
 |L_\rho([x])-P_A^NL_\rho([x])| \leq 
 \int |L_\rho([x])-L_\rho([A^N(\omega)x])|\,d\mathbb{P}\leq \log C_N.
$$
Thus, this bound and~\eqref{eq:error-LG-new}
imply the first inequality in~\eqref{eq:error-outside-new2} on $X$, and continuity of
\eqref{eq:error-Dt} extends the bound to $\proj(E)$.

Now let $r([x])\ge t$ and choose $D_0$ so that $3D_0^{2/p}\le1$.
Then, since $\rho = tD^{2/p}$ with $D<D_0$, it holds that $t\ge3\rho$, hence $G_\rho([x])=0$. Since $G_\rho$ is supported in
$\{r\le3\rho\}\cap X$, where $1\le(3\rho)^pR$,  we have that $|G_\rho| \leq \|u\|_\infty \leq \|u\|_\infty(3\rho)^pR$ on $\{r\leq 3\rho\} \cap X$. Thus,  $|G_\rho|\leq C \rho^p R$ on $X$. Then, Lemma~\ref{lem:drift} gives 
\begin{align*}
 |P_A^NG_\rho([x])|
 \le C\rho^pP_A^NR([x])
 \le C\rho^p(R([x])+1) \le C\left[\left(\frac\rho t\right)^p+\rho^p\right].
\end{align*}
Since $(\rho/t)^p=D^2$ and $\rho^p=t^pD^2\le D^2$, we obtain
\begin{equation}\label{eq:L-out-D2}
 |G_\rho([x])-P_A^NG_\rho([x])|\le CD^2 \quad \text{for $r([x])\ge t$}.
\end{equation}
For the logarithmic correction,
$0\le L_\rho([x])
 \le\frac1p\rho^pR([x])
 \le C\left(\frac\rho t\right)^p
 =CD^2$. Using Jensen's inequality and the drift once more,
\[
\begin{aligned}
 P_A^NL_\rho([x])
 &\le\frac1p\log(1+\rho^pP_A^NR([x]))\\
 &\le C\rho^p(R([x])+1)
 \le C\left[\left(\frac\rho t\right)^p+\rho^p\right]\le CD^2.
\end{aligned}
\]
Thus $|L_\rho([x])-P_A^NL_\rho([x])|\le CD^2$.
Combining this with~\eqref{eq:L-out-D2} in \eqref{eq:error-LG-new} proves~\eqref{eq:error-outside-new2}.
\end{proof}

\begin{proof}[Proof of Proposition~\ref{prop:defect-mass-new}] 
Since $m_k$ is maximizing,
$N(\lambda_1(B_k)-\lambda)=\int(\phi_{B_k,N}-N\lambda)\,dm_k$.
Using \eqref{eq:error-Dt} and $P_{B_k}^N$-stationarity, we obtain
\[
\begin{aligned}
 -ND_k
 &=\int e_{D_k,t}\,dm_k+\int(h_{D_k,t}-P_A^Nh_{D_k,t})\,dm_k
   +\int(\phi_{B_k,N}-\phi_{A,N})\,dm_k\\
 &=\int e_{D_k,t}\,dm_k+\int(P_{B_k}^N-P_A^N)h_{D_k,t}\,dm_k
   +\int(\phi_{B_k,N}-\phi_{A,N})\,dm_k.
\end{aligned}
\]
Lemma~\ref{lem:fixed-pert} and \eqref{eq:gDt-norm} give
\begin{equation}\label{eq:defect-before-split-new}
 ND_k\le-\int e_{D_k,t}\,dm_k+C\delta_k
          +C\delta_k^\alpha t^{-1}D_k^{-2/p}.
\end{equation}
For $0<\delta_k,D_k,t<1$, we have
$\delta_k\le\delta_k^\alpha t^{-1}D_k^{-2/p}$, so the $C\delta_k$ term is absorbed
into the last term. Lemma~\ref{lem:regularized-error-new} also gives
\[
 -\int_{\{r<t\}}e_{D_k,t}\,dm_k\le Cm_k(\{r<t\}),\qquad
 -\int_{\{r\ge t\}}e_{D_k,t}\,dm_k\le CD_k^2.
\]
Substituting into \eqref{eq:defect-before-split-new} and dividing by $N$ yields
\begin{equation}\label{eq:defect-mass-new}
 D_k\le Cm_k(\{r<t\})+CD_k^2+C\delta_k^\alpha t^{-1}D_k^{-2/p}.
\end{equation}
Since $\alpha-2a>0$, by~\eqref{eq:superpoly-calculus-new}, for $t\ge\delta_k^{2a}$ it holds
$\delta_k^\alpha t^{-1}D_k^{-2/p}
 \le\delta_k^{\alpha-2a}D_k^{-2/p}=o(D_k)$.
Also $D_k^2=o(D_k)$ because $D_k\to0$. Thus, 
uniformly for $\delta_k^{2a}\le t\le\delta_k^a$, the two errors in \eqref{eq:defect-mass-new}
are at most $D_k/2$ for large $k$, and
$m_k(\{r<t\})\ge D_k/(2C)$.  For the upper bound, use \eqref{eq:tube-upper-new} with
$\varepsilon=t/3$.  Hence, $m_k(\{r<t\})
 \le m_k(\{r\le t\})
 \le C(D_k+t^p+\delta_k^\alpha t^{-1})$.
Now $t^p\le\delta_k^{ap}=o(D_k)$ by~\eqref{eq:superpoly-calculus-new}, and
$\delta_k^\alpha t^{-1}
 \le\delta_k^{\alpha-2a}=o(D_k)$.
Thus  $
 m_k(\{r\le t\})\le CD_k$
 for all $t\in[\delta_k^{2a},\delta_k^a]$.  Combining this upper bound with the lower bound completes the proof of~\eqref{eq:radial-comparable-new}. 
\end{proof}

\subsection{Selection of nested tubes around the equator}
\label{sec:annulus-selection}
In Proposition~\ref{prop:defect-mass-new}, we computed the mass of the $P_{B_k}$-stationary measure $m_k$ on a neighborhood  $\{r \leq t\}$ of the projectivization of the equator for $t \in [\delta_k^{2a}, \delta_k^a]$. Because this neighborhood depends on $k$, we cannot directly apply~\cite[Thm.~7.1]{AEV}, see Theorem~\ref{prop:aev-induction} below. To overcome this, we will apply a linear rescaling $C_k$ that maps a selected nested sequence of cone-tube neighborhoods $ \mathcal{T}({t'_k}) \subset \mathcal{T}({t_k})$ of $\proj(E)$ in $\{r \leq t_k\}$ with $\delta_k^{2a} < t'_k < t_k < \delta_k^a$ to the fixed neighborhood $\{r \leq 1/\sqrt{2}\}$. Namely $C_k\mathcal{T}({t_k})=\{r\leq 1/\sqrt{2}\}$  and $C_k\mathcal{T}({t'_k})=\{r\leq (1+(t_k/t'_k)^2)^{-1/2}\}$ choosing carefully $t_k/t'_k \to \infty$ as $k\to \infty$. 
Conjugating the cocycles $B_k$ by this linear rescaling yields a new sequence of stationary probability measures $\tilde{m}_k$ whose normalized restriction to $\{r\leq 1/\sqrt{2}\}$ assigns uniform mass to $\{r\leq (1+(t_k/t'_k)^2)^{-1/2}\}$ for every $k\geq 1$. Here we construct these nested cone-tube neighborhoods.

Write $x=e+v$ according to
$\RR^d=E\oplus E^\perp$ and, on $\proj(\RR^d)\setminus\proj(E^\perp)$, define
\[
 s([e+v]):=\frac{\|v\|}{\|e\|},
 \qquad
 \mathcal T(t):=\{[e+v]:s([e+v])\le t\}.
\]
By orthogonality,
\begin{equation}\label{eq:r-slope-new}
 r([e+v])= \frac{\|v\|}{\|e+v\|} = \frac{\|v\|}{\sqrt{\|e\|^2+\|v\|^2}}=\frac{s([e+v])}{\sqrt{1+s([e+v])^2}}.
\end{equation}
Since the function $r=f(s)=s(1+s^2)^{-1/2}$ is  strictly increasing for $s\geq 0$, we have  $s\leq t \iff r=f(s)\leq f(t)={t}(1+t^2)^{-1/2}$. Thus, $t 2^{-1/2} \leq t ( {1+t^2})^{-1/2} \leq t$ for $0<t\leq 1$ implies 
\begin{equation}\label{eq:radial-slope-compare-new}
 \{r\le t/\sqrt2\}\subset\mathcal T(t)\subset\{r\le t\}.
\end{equation}
Therefore, \eqref{eq:radial-comparable-new} gives
\begin{equation}\label{eq:slope-comparable-new}
 cD_k\le m_k(\mathcal T(t))\le CD_k \quad \text{for \  $2\delta_k^{2a}\le t\le\frac12\delta_k^a$}.
\end{equation}

\begin{lemma} \label{lem:three-concrete-inputs}
After discarding finitely many terms, there exist sequences $(t_k)_{k\geq 1}$ and $(t'_k)_{k\geq 1}$ with $2\delta_k^{2a}\le t'_k<t_k \le\frac12\delta_k^a$ such that
\begin{equation}
    \label{eq:three-concrete-inputs}
 \frac{\delta_k}{t_k}\longrightarrow0,\quad  \frac{t_k}{t'_k} \to \infty \quad  \text{and} \quad
 \frac{m_k(\mathcal T(t'_k))}{
 m_k(\mathcal T(t_k))} \to 1 \quad \text{as $k\to \infty$.}
\end{equation}
\end{lemma}

\begin{proof}
After discarding finitely many terms, set
$K_k:=\log(1/\delta_k)>1$. Then $K_k\to\infty$. Starting from the upper end of the interval in
\eqref{eq:slope-comparable-new}, consider the radii
\[
 t_{k,j}:=\frac12\delta_k^aK_k^{-j},
 \qquad 0\le j\le L_k,
\]
where $L_k$ is the largest integer such that
$t_{k,L_k}\ge2\delta_k^{2a}$. By maximality,
$K_k^{L_k+1}>\frac14\delta_k^{-a}$,
and therefore
\[
 L_k>
 \frac{a\log(1/\delta_k)-\log4}
      {\log\log(1/\delta_k)}-1
 \longrightarrow\infty.
\]
On the other hand,  the annuli $\mathcal T(t_{k,j})\setminus\mathcal T(t_{k,j+1})$, 
$0\le j<L_k$, are pairwise disjoint and contained in $\mathcal T(t_{k,0})$. Hence,
by \eqref{eq:slope-comparable-new},
\[
 \sum_{j=0}^{L_k-1}
 m_k\bigl(
 \mathcal T(t_{k,j})\setminus\mathcal T(t_{k,j+1})
 \bigr)
 \le CD_k.
\]
Thus, $m_k(
 \mathcal T(t_{k,j_k})\setminus
 \mathcal T(t_{k,j_k+1}))
 \le CD_k / L_k$ for some $j_k<L_k$. 
Set $t_k:=t_{k,j_k}$ and $t'_k:=t_{k,j_k+1}$. Clearly $t'_k<t_k$ and since both radii belong to the interval of
\eqref{eq:slope-comparable-new}, $m_k(\mathcal T(t_k))\ge cD_k$. Consequently,
\[
\begin{aligned}
 0
 \le
 1-\frac{m_k(\mathcal T(t'_k))}
          {m_k(\mathcal T(t_k))} =
 \frac{
 m_k(\mathcal T(t_k)\setminus\mathcal T(t'_k))
 }{
 m_k(\mathcal T(t_k))
 }
 \le \frac{C}{cL_k}
 \longrightarrow0.
\end{aligned}
\]
This proves the last limit in \eqref{eq:three-concrete-inputs}. Moreover,
$t_k/{t'_k}=K_k\longrightarrow\infty$.
Finally, since $t_k\ge2\delta_k^{2a}$,
$\delta_k/t_k
 \le\frac12\delta_k^{1-2a}
 \longrightarrow0$,
because $a<\alpha/4\le1/8$. This proves the lemma.
\end{proof}

\subsection{The renormalization}

For $t>0$ define the linear map
\begin{equation}\label{eq:Ct-new}
 C_t(e+v):=e+t^{-1}v,
 \qquad e\in E,\ v\in E^\perp.
\end{equation}
Since $C_t[e+v]=[e+t^{-1}v]$ with $e\in E$ and $t^{-1}v\in E^\perp$, $s(C_t[e+v])=\frac{t^{-1}\|v\|}{\|e\|} = \frac{1}{t} s([e+v])$.  
Thus, \eqref{eq:r-slope-new} gives the identity
\begin{equation*}%\label{eq:r-after-Ct-new}
 r(C_t[e+v])
 =\frac{s([e+v])/t}{\sqrt{1+(s([e+v])/t)^2}}.
\end{equation*}
Hence, for any $0<t'\leq t$ it holds that 
\begin{align}
 C_t\mathcal T(t')=\bigg\{ r \leq \frac1{\sqrt{1+(t/t')^2}} \bigg\}.     \label{eq:inner-fixed-new}
\end{align}
Define the conjugacy  
\begin{equation}\label{eq:Btilde-new}
 \tilde B_k:=C_{t_k}B_kC_{t_k}^{-1}, \qquad \tilde\nu_k := (\tilde B_k)_* \mathbb{P}.
\end{equation}
We also define the block-diagonal reference cocycle
\begin{equation}\label{eq:A0-new}
 A_0(\omega):=
 \begin{pmatrix}A(\omega)|_E&0\\0&\bar A(\omega)\end{pmatrix}
\end{equation}
where we recall that $\bar{A}$ denoted the quotient cocycle on $V$ which is isometrically identified with $E^\perp$. We have that $\lambda_1(A_0)=\lambda$ since the restriction to $E$ has top exponent
$\lambda_1(A|_E)<\lambda$, while the quotient block has top exponent $\lambda$, cf.~\cite[Claim 5.5.2]{BMbook}.  Moreover, the equator of $\nu_{A_0}=(A_0)_*\mathbb P$ is exactly $E$.
Indeed, $E$ is invariant and has top exponent smaller than $\lambda$.
If an $A_0$-invariant subspace $W$ is not contained in $E$, then $QW\ne0$
is $\bar A$-invariant. Quasi-irreducibility of $\bar A$ gives
$\lambda_1(\bar A|_{QW})=\lambda$. Since this restriction is a quotient
of $A_0|_W$, its top exponent is at most $\lambda_1(A_0|_W)\le\lambda$.
Thus $W$ cannot be contained in the equator.

\begin{lemma}\label{lem:renorm}
The law $\tilde{\nu}_k$ is generic and $\tilde{\nu}_k \to \nu_{A_0}$ as $k\to \infty$ in the weak--Hausdorff topology. 
\end{lemma}
\begin{proof}
Since conjugation by $C_{t_k}$ is an algebraic automorphism of $G$,
the inverse image of every proper algebraic subvariety of $G$ is
again a proper algebraic subvariety. Thus, the genericity of
$\nu_k$ implies that of $\tilde\nu_k$.

Relative to $\RR^d=E\oplus E^\perp$, write
\[
 C_t=
 \begin{pmatrix}I_E&0\\0&t^{-1}I_{E^\perp}\end{pmatrix},
 \qquad
 A(\omega)=
 \begin{pmatrix}a_{11}&a_{12}\\0&a_{22}\end{pmatrix},
 \qquad
 B_k(\omega)=
 \begin{pmatrix}b_{11}&b_{12}\\b_{21}&b_{22}\end{pmatrix}.
\]
Here $a_{11}=A(\omega)|_E$ and $a_{22}=\bar A(\omega)$.
A direct computation gives
\[
 \tilde B_k(\omega)
 =
 \begin{pmatrix}
 b_{11}&t_kb_{12}\\
 t_k^{-1}b_{21}&b_{22}
 \end{pmatrix}.
\]
Since the supports of $A$ and $B_k$ lie in a common compact
subset of $G$, their matrix blocks are uniformly bounded.
Moreover, $a_{21}=0$, so
\[
 \|\tilde B_k(\omega)-A_0(\omega)\|
 \le C\left(
 t_k+\frac{\|B_k(\omega)-A(\omega)\|}{t_k}
 \right).
\]
The same computation applied to $B_k^{-1}$ and $A^{-1}$
gives
\[
 \|\tilde B_k(\omega)^{-1}-A_0(\omega)^{-1}\|
 \le C\left(
 t_k+\frac{\|B_k(\omega)^{-1}-A(\omega)^{-1}\|}{t_k}
 \right),
\]
because $A^{-1}$ also preserves $E$. Integrating both
inequalities yields
$d(\tilde B_k,A_0)
 \le C (t_k+\frac{d(A,B_k)}{t_k})$.

The same block estimates hold for any
$a\in\supp\nu_A$ and $b\in\supp\nu_k$. Applying them in both
directions to the supports gives
$
 \delta_H(\supp\tilde\nu_k,\supp\nu_{A_0})
 \le C (
 t_k+\frac1{t_k}\delta_H(\supp\nu_k,\supp\nu_A))$.
Consequently,
\[
 \Delta(\tilde B_k,A_0)
 \le C\left(t_k+\frac{\delta_k}{t_k}\right)
 \longrightarrow0 \quad \text{as $k\to\infty$},
\]
where the convergence follows from
\eqref{eq:three-concrete-inputs} and $t_k\to0$.
By~\eqref{eq:law-nonexpansion}, this implies
$\delta_{\mathcal T}(\tilde\nu_k,\nu_{A_0})\to0$ and therefore $\tilde{\nu}_k$ converges to $\nu_{A_0}$ in the weak--Hausdorff
topology. 
\end{proof}

\subsection{The localized induction}
We recall some definitions from \cite[Sec.~6.5]{AEV}. For a probability
$\mu$ on $G$ acting on an algebraic variety $Z$, write the transition probability
$\mu_x=(g\mapsto gx)_*\mu$. A transition kernel $x\mapsto\sigma_x$ is a family of probability
measures such that $x\mapsto\sigma_x(F)$ is measurable for every Borel set $F$.
A probability on $Z$ is generic if it gives
zero mass to every proper algebraic subvariety. For a compact neighborhood
$U\subset Z$, fix $a_U>0$ and set
$U_s=\{x\in U:d(x,Z\setminus U)\ge a_Us\}$ for $0\le s\le2$.
The core, border and sharp core are, respectively,
\begin{align*}
 \Chi_\mu U
 &=\{x\in U_2:gx\in U_2\text{ for every }g\in\supp\mu\},\\
 \partial_\mu U&=U\setminus\Chi_\mu U,\\
 \Chi_\mu^\# U
 &=\{x\in\Chi_\mu U:g^{-1}x\in\Chi_\mu U
                       \text{ for every }g\in\supp\mu\}.
\end{align*}
A transition kernel $\sigma_x$ on $U$ defines a \emph{Markov operator}
$T:B(U)\to B(U)$ by
\[
 T\varphi(x)=\int_U \varphi(y)\,d\sigma_x(y).
\]  
 Here $B(U)$ denotes the space of real-valued bounded Borel functions on $U$. This operator $T$ is said to be adapted to $(\mu,U)$ if
\begin{enumerate}[label=\rm(\alph*)]
\item $\sigma_x=\mu_x$ in a neighborhood of $\Chi_\mu U$;
\item $\sigma_x(\Chi_\mu^\#U)=0$ for every $x\in\partial_\mu U$;
\item $\sigma_x$ is generic for every $x\in U$.
\end{enumerate}
Here and below, $T$ is said to be \emph{Feller continuous} if it preserves
continuous functions. 

For $1\le r\le d$, let $\Gr(r,\mathbb{R}^d)$ be the Grassmannian of $r$-dimensional linear subspaces of $\mathbb{R}^d$ endowed with the metric
\[
 d_r(L,L')=\sup_{u\in L\setminus\{0\}}
          \inf_{v\in L'\setminus\{0\}}|\sin\angle(u,v)|.
\]
Put $E_r=\{L\in\Gr(r,\RR^d):L\subset E\}$ and denote its closed
$\varepsilon$-neighborhood by $E_r(\varepsilon)$.
% For $U=E_r(\varepsilon)$, take $a_U=\varepsilon/100$, as in
% \cite[Example~6.24 and Sec.~7]{AEV}.
Notice that $E_1=\proj(E)$ and $E_1(\varepsilon)=\{r\leq \varepsilon\}$.  The following is the main inductive statement of Avila--Eskin--Viana~\cite{AEV}. Here $\mu^{(n)}$ denotes the law of a product of $n$ independent matrices with law $\mu$.

\begin{theorem}[{\cite[Thm.~7.1]{AEV}}]\label{prop:aev-induction}
Let $\mu_k$ be generic compactly supported laws on $G$ with Zariski-dense
supports, converging in the weak--Hausdorff topology to $\mu_\infty$,
and let $E$ be the nonzero equator of $\mu_\infty$.
Let $1\le r\le d$ and suppose that there exist
$\varepsilon_r>0$ and $n_r\in\mathbb N$ such that
\begin{enumerate}[label=\rm(\roman*)]
\item For each $k\in\mathbb N$ there is a Feller continuous Markov operator
\[
 T_{k,r}: B(E_r(\varepsilon_r))\longrightarrow
           B(E_r(\varepsilon_r)),\qquad
 T_{k,r}\varphi(x)=\int_{E_r(\varepsilon_r)}
                          \varphi(y)\,d\sigma_{k,r,x}(y),
\]
adapted to $(\mu_k^{(n_r)},E_r(\varepsilon_r))$.
\item For each $k\in\mathbb N$ there is a $T_{k,r}$-invariant probability
$\eta_{k,r}$ on $E_r(\varepsilon_r)$ such that \\
$\eta_{\infty,r}=\lim_k\eta_{k,r}$ exists and satisfies
$\eta_{\infty,r}(E_r)>0$.
\end{enumerate}
If $r<\dim E$, there exist $\varepsilon_{r+1}>0$ and $n_{r+1}\in\mathbb N$
such that \textup{(i)} and \textup{(ii)} hold with $r$ replaced by $r+1$.
If $r=\dim E$, then \textup{(i)} and \textup{(ii)} cannot hold.
\end{theorem}

\subsection{The lower bound}
Let $\tilde\nu_k$ and $\nu_{A_0}$ be, respectively, the laws of $\tilde B_k$ in~\eqref{eq:Btilde-new} and $A_0$ in~\eqref{eq:A0-new}. Take $\varepsilon_1=2^{-1/2}$ and $n_1=1$ and set
\begin{align*}
\mu_k=\tilde \nu_k, \quad  \mu_\infty=\nu_{A_0}, 
     \quad
 U=E_1(\varepsilon_1), \quad
 \eta_{k,1}=\frac{\tilde m_k|_U}{\tilde m_k(U)} \ \ \text{with $\tilde m_k=(C_{t_k})_*m_k$}.
\end{align*}
By Lemma~\ref{lem:renorm}, the laws $\mu_k$ are generic and converge in the weak-Hausdorff topology to the law $\mu_\infty$, whose equator is
$E$. The measure $\tilde m_k$ is $\mu_k$-stationary. Moreover, by~\eqref{eq:inner-fixed-new},
$C_{t_k}\mathcal{T}(t_k)=\{r\leq 2^{-1/2}\}=E_1(\varepsilon_1)$. Hence, by  \eqref{eq:slope-comparable-new},
\[
 \tilde m_k(U)=m_k(\mathcal T(t_k))\ge cD_k>0. 
\]
so $\eta_{k,1}$ is well defined.  The localization criterion \cite[Prop.~6.25]{AEV} applies to a
Markov operator $T$ on a larger space $Z$, an invariant probability $m$,
and a subset $U$ with $m(U)>0$. Its hypotheses on the transitions are
\[
 \sigma_x=\mu_x\ (x\in U),\qquad
 \sigma_x(\Chi_\mu^\#U)=0\ (x\notin U),\qquad
 \sigma_x\text{ generic }\ (x\in Z).
\]
Together with \cite[Prop.~6.15]{AEV}, it gives a continuous operator
on $U$, adapted to $(\mu,U)$, preserving the normalized restriction  $m|_U/m(U)$.
In our application $Z=\proj(\RR^d)$, $T=P_{\tilde{B}_k}$,
$m=\tilde m_k$ and $\mu=\mu_k=\tilde{\nu}_k$.
The first condition holds identically. For the second, if $x\notin U$
and $gx\in\Chi_{\mu_k}^\#U$ for some $g\in\supp\mu_k$, the definition of the
sharp core would imply $x=g^{-1}(gx)\in\Chi_{\mu_k} U\subset U$,
a contradiction. Thus $\mu_{k,x}(\Chi_{\mu_k}^\#U)=0$.
For the third, let $Z_0$ be a proper algebraic subvariety of $\proj(\RR^d)$.
The orbit map $g\mapsto gx$ is algebraic and surjective, so its inverse
image of $Z_0$ is a proper algebraic subvariety of $G$.
Genericity of $\mu_k=\tilde\nu_k$ provided by Lemma~\ref{lem:renorm} therefore gives
$\mu_{k,x}(Z_0)=0$.
All three hypotheses hold. We obtain Feller continuous Markov operators
$ T_{k,1}: B(U)\longrightarrow B(U)$
adapted to $(\mu_k,U)$ and preserving~$\eta_{k,1}$.

By compactness of $U$, pass to a subsequence such that
$\eta_{k,1}\to\eta_{\infty,1}$ in the weak$^*$ topology. For fixed $\varepsilon>0$, the sets
$\{r\le(1+(t_k/t'_k)^2)^{-1/2}\}$ are eventually contained in the closed set
$E_1(\varepsilon)$, since $t_k/t'_k\to\infty$. Thus Portmanteau's theorem gives
\[
\begin{aligned}
 \eta_{\infty,1}(E_1(\varepsilon))
 &\ge\limsup_k\eta_{k,1}(E_1(\varepsilon))\\
 &\ge\lim_k\eta_{k,1}\bigl(\{r\le(1+(t_k/t'_k)^2)^{-1/2}\}\bigr)=\lim_k\frac{m_k(\mathcal T(t'_k))}{m_k(\mathcal T(t_k))}=1,
\end{aligned}
\]
where the last line follows from \eqref{eq:inner-fixed-new} and
Lemma~\ref{lem:three-concrete-inputs}. Continuity from above as
$\varepsilon\downarrow0$ implies $\eta_{\infty,1}(E_1)=1$.
Theorem~\ref{prop:aev-induction} therefore applies.
Starting at $r=1$ and applying its conclusion successively yields the
forbidden case $r=\dim E$. This is a contradiction.
Consequently there are $C_->0$, $\theta_->0$ and $r_->0$ such that
\begin{equation} \label{eq:lower-holder-new}
     \lambda_1(B)\ge\lambda-C_-\Delta(A,B)^{\theta_-}
 \quad \text{whenever $\Delta(A,B)<r_-$}.
\end{equation}
\section{Proofs of the main results}\label{sec:conclusion} 

\enlargethispage{0.5cm}
\subsection{Proof of Theorem~\ref{thm:main}} Combining \eqref{eq:upper-holder-new} with~\eqref{eq:lower-holder-new} proves
\eqref{eq:fixed-reference-target}, after taking
$\theta=\min\{\theta_+,\theta_-,1\}$. %and decreasing the radius below $1$.
Finally, for each sufficiently close law $\nu$, use
\eqref{eq:BM-realization} to obtain
\[
 |\lambda_1(\nu)-\lambda_1(\nu_0)|
 \le C\bigl(\delta_{\cT}(\nu,\nu_0)+\zeta\bigr)^\theta.
\]
Letting $\zeta\downarrow0$ completes the proof of
Theorem~\ref{thm:main}.

\subsection{Proof of Corollary~\ref{cor:first-k}}
For $1\le j\le d$, put $\rho_j(g)=\wedge^j g$,
$\nu^{[j]}=(\rho_j)_*\nu$, and
$S_j(\nu)=\sum_{i=1}^j\lambda_i(\nu)$; set $S_0=0$.
Equip each exterior power with its induced Euclidean norm and denote
by $\delta_{\cT}^{[j]}=\delta_W^{[j]}+\delta_H^{[j]}$ the corresponding
Wasserstein--Hausdorff distance.

For $\nu$ in a sufficiently small $\delta_{\cT}$-neighborhood of $\nu_0$,
the supports of $\nu$ and $\nu_0$ lie in a common compact set $K\subset G$.
Choose $M\ge1$ such that $\|g\|,\|g^{-1}\|\le M$ on $K$.
The tensor-product telescoping identity, restricted to the alternating
subspace, gives
\[
 \|\rho_j(g)-\rho_j(h)\|\le jM^{j-1}\|g-h\|,
 \qquad g,h\in K.
\]
Since $\rho_j(g)^{-1}=\rho_j(g^{-1})$, the same estimate for inverses yields
\[
 \|\rho_j(g)-\rho_j(h)\|
 +\|\rho_j(g)^{-1}-\rho_j(h)^{-1}\|
 \le jM^{j-1}\delta(g,h).
\]
Let $\Pi(\nu,\nu_0)$ denote the set of couplings of $\nu$ and $\nu_0$,
that is, probability measures $\pi$ on $K\times K$ whose first and second
marginals are $\nu$ and $\nu_0$, respectively. By definition,
\[
 \delta_W(\nu,\nu_0)
 =\inf_{\pi\in\Pi(\nu,\nu_0)}\int_{K\times K}\delta(g,h)\,d\pi(g,h).
\]
For every such $\pi$, the push-forward $(\rho_j\times\rho_j)_*\pi$ is a
coupling of $\nu^{[j]}$ and $\nu_0^{[j]}$. Therefore,
\[
\begin{aligned}
 \delta_W^{[j]}(\nu^{[j]},\nu_0^{[j]})
 &\le\int_{K\times K}\bigl(\|\rho_j(g)-\rho_j(h)\|
       +\|\rho_j(g)^{-1}-\rho_j(h)^{-1}\|\bigr)\,d\pi(g,h)\\
 &\le jM^{j-1}\int_{K\times K}\delta(g,h)\,d\pi(g,h).
\end{aligned}
\]
Taking the infimum over $\pi$ gives
$\delta_W^{[j]}(\nu^{[j]},\nu_0^{[j]})\le jM^{j-1}\delta_W(\nu,\nu_0)$.
The same pointwise bound, applied in both directions to the supports,
gives the corresponding Hausdorff estimate, since
$\supp\nu^{[j]}=\rho_j(\supp\nu)$. Adding the two estimates yields
\[
 \delta_{\cT}^{[j]}(\nu^{[j]},\nu_0^{[j]})
 \le jM^{j-1}\delta_{\cT}(\nu,\nu_0).
\]

The Lyapunov exponents of the $j$th exterior power are the sums of
$j$ distinct indexed exponents of the original product, as follows
from the corresponding identity for singular values. Thus, for
$j=1,\dots,\min\{k,d-1\}$,
\[
 \lambda_1(\nu^{[j]})=S_j(\nu),\qquad
 \lambda_1(\nu_0^{[j]})-\lambda_2(\nu_0^{[j]})
 =\lambda_j(\nu_0)-\lambda_{j+1}(\nu_0)>0.
\]
Theorem~\ref{thm:main} and the preceding 
estimate therefore give, for each $j=1,\dots,\min\{k,d-1\}$,
\[
 |S_j(\nu)-S_j(\nu_0)|
 \le C_j\delta_{\cT}(\nu,\nu_0)^{\theta_j}
\]
for $\nu$ in a neighborhood of $\nu_0$, with $\theta_j\in(0,1]$.

If $k=d$, multiplicativity of the determinant gives
\[
 S_d(\nu)
 =\lim_{n\to\infty}\frac1n\int_{G^n}\log|\det(g_n\cdots g_1)|
     \,d\nu^{\otimes n}
 =\int\log|\det g|\,d\nu(g).
\]
For $g,h\in K$, their singular values lie in $[M^{-1},M]$ and satisfy
$|\sigma_i(g)-\sigma_i(h)|\le\|g-h\|$. Consequently,
\[
 |\log|\det g|-\log|\det h||
 \le M\sum_{i=1}^d|\sigma_i(g)-\sigma_i(h)|
 \le dM\|g-h\|.
\]
For every $\pi\in\Pi(\nu,\nu_0)$, its marginal identities give
\[
\begin{aligned}
 S_d(\nu)-S_d(\nu_0)
 &=\int_{K\times K}\log|\det g|\,d\pi(g,h)
   -\int_{K\times K}\log|\det h|\,d\pi(g,h)\\
 &=\int_{K\times K}\bigl(\log|\det g|-\log|\det h|\bigr)\,d\pi(g,h).
\end{aligned}
\]
Taking absolute values and using the preceding pointwise estimate yields
\[
 |S_d(\nu)-S_d(\nu_0)|
 \le dM\int_{K\times K}\|g-h\|\,d\pi(g,h)
 \le dM\int_{K\times K}\delta(g,h)\,d\pi(g,h).
\]
Since this holds for every coupling, taking the infimum gives
$|S_d(\nu)-S_d(\nu_0)|\le dM\delta_W(\nu,\nu_0)$.
Finally, take a common radius $r_0<1$ and
$\theta=\min_{1\le j\le\min\{k,d-1\}}\theta_j$.
Since $\lambda_j=S_j-S_{j-1}$, the triangle inequality yields the
claimed estimate after increasing $C$.

\section*{Acknowledgments}
P.~G.~Barrientos was partially supported by the Spanish MCIN under grant
PID-2023-147461NB-I00, by CNPq under grants PQ~302738/2025-8 and Projeto Universal 401737/2025-0, and by CAPES---Finance Code 001. AI was used to assist with language editing, organization, and the checking and refinement of mathematical arguments, especially of Section~\ref{sec:bad-sequence}. The author
reviewed the resulting text and took responsibility for the mathematical
content and the final manuscript.

% Retain the author's custom style when available.
\bibliographystyle{alpha3}
\bibliography{references}
\end{document}